\documentclass{article}
\usepackage{geometry}               
\usepackage[english]{babel}
\usepackage[utf8]{inputenc}
\usepackage{amssymb,amsmath,amsthm,amsfonts}
\usepackage{graphicx}
\usepackage[usenames,dvipsnames]{xcolor}
\usepackage{cancel}
\usepackage[colorlinks]{hyperref}
\usepackage{enumitem}
\usepackage{comment}

\newcommand{\N}{\mathbb{N}}

\newcommand{\R}{\mathbb{R}}

\newcommand{\sph}{\mathbb{S}}

\newcommand{\eps}{\varepsilon}

\def\XXint#1#2#3{{\setbox0=\hbox{$#1{#2#3}{\int}$ }
\vcenter{\hbox{$#2#3$ }}\kern-.6\wd0}}

\newtheorem{proposition}{Proposition}[section]
\newtheorem{theorem}[proposition]{Theorem}
\newtheorem{corollary}[proposition]{Corollary}
\newtheorem{lemma}[proposition]{Lemma}

\theoremstyle{definition}

\newtheorem{remark}[proposition]{Remark}

\numberwithin{equation}{section}
\newcommand{\beq}{\begin{equation}}
\newcommand{\eeq}{\end{equation}}
\newcommand{\ben}{\begin{enumerate}}
\newcommand{\een}{\end{enumerate}}
\newcommand{\bit}{\begin{itemize}}
\newcommand{\eit}{\end{itemize}}

\newcommand{\vol}{{\mathrm{vol}}}
\newcommand{\dvol}{\,d\vol_{M^{k}_y}}
\newcommand{\thresh}[1]{\rho^*_{\mathrm{#1}}}
\newcommand{\Beta}{\mathbf{B}}

\title{Energy local minimizers for the\\ 
nonlinear Schrödinger equation\\ 
on product spaces}

\author{Dario Pierotti, Gianmaria Verzini and Junwei Yu}

\date{\today}

\begin{document}
\maketitle

\begin{abstract}
We investigate the existence of local minimizers with prescribed $L^{2}$-norm for the energy functional 
associated to the mass-supercritical nonlinear Schrödinger equation on the product space 
$\R^{N} \times M^{k}$, where $(M^{k},g)$ is a compact Riemannian manifold, thus complementing 
the study of the mass-subcritical case performed by Terracini, Tzvetkov and Visciglia in 
[\emph{Analysis \& PDE} 7(1):73–96, 2014]. 

First we prove that, for small 
$L^{2}$-mass, the problem admits local minimizers. Next, we show that when the $L^{2}$-norm is 
sufficiently small, the local minimizers are constants along $M^k$, and they coincide with those of the 
corresponding problem on $\R^{N}$. Finally, under certain conditions, we show that the local minimizers 
obtained above are nontrivial along $M^{k}$. The latter situation occurs, for instance, for every $M^{k}$ of 
dimension $k\ge 2$, with the choice of an appropriate metric $\hat g$, and in $\R\times \sph^k$, $k\ge3$, where $\sph^k$ is endowed with the standard round metric.
\end{abstract}
\noindent
{\footnotesize \textbf{AMS-Subject Classification}}.
{\footnotesize 58J05, 35Q55; 58J70.}\\
{\footnotesize \textbf{Keywords}}.
{\footnotesize Normalized solutions, elliptic PDEs on manifolds, constrained critical points, waveguide manifolds.}

\section{Introduction}

Let $(M^k,g)$ be a smooth compact Riemannian manifold (without boundary), with $\dim M^k=k\ge1$, and let $N\ge1$ be an 
integer. Throughout the paper, we assume without loss of generality (up to a scale factor) that 
\begin{equation}\label{eq:vol}
\vol(M^{k})=1.
\end{equation}

A decade ago, in the paper \cite{MR3219500}, Terracini, Tzvetkov and Visciglia considered the nonlinear 
Schrödinger equation posed on the product space $\R^{N}\times M^{k}$ (endowed with 
the product metric induced by the Euclidean one on $\R^N$, and $g$):
\begin{equation}\label{eq:evoNLS}
\begin{cases}
i\partial_t u - \Delta_{x,y} u - u|u|^{{\alpha}}=0, \ \ \ (t,x,y)\in \R\times \R_{x}^{N} \times M_{y}^{k},\\
u(x,y,t)=u_0(x,y),
\end{cases}
\end{equation}
where
\[
\Delta_{x,y}=\Delta_{x}+\Delta_{y}=\sum_{i=1}^N \partial_{x_ix_i} + \Delta_{y},
\]
and $\Delta_{y}$ is the Laplace-Beltrami operator on $(M^k,g)$. In particular, they were interested in studying 
the nature of the ground states of  \eqref{eq:evoNLS} in the mass subcritical regime
\[
0 < \alpha < \frac{4}{N+k}.
\]
As they show, in such case the energy functional $E: H^{1}(\R^{N}\times M^{k},\mathbb{C}) \to \R$, defined by
\begin{equation}\label{fun:main_xy}
E(u)= \int_{M_{y}^{k}}\int_{\R_{x}^{N}}\left(\frac{1}{2}|\nabla_{x,y}u|^{2}-\frac{1}{2+\alpha}|u|^{2+\alpha}\right)\,dx  \dvol,
\end{equation}
is bounded below on the $L^{2}$-mass constraint
\begin{equation}\label{manifold:rho}
S_{\rho}= \left\{ u \in H^{1}(\R^{N}\times M^{k}): \int_{M_{y}^{k}}\int_{\R^{N}_{x}} |u|^{2}\,dx  \dvol= \rho^{2}\right\},
\end{equation}
so one can try to construct orbitally stable solitary waves by solving, for some $\rho$, the (global) minimization problem
\begin{equation}\label{eq:GS_prod_subc}
K_\rho:=\inf\{E(u): u \in S_\rho\}.
\end{equation}
This eventually leads to a solution $u=u(x,y)\in H^{1}(\R^{N}\times M^{k})$, with suitable Lagrange multiplier 
$\omega\in\R$, of the normalized problem
\begin{equation}\label{eq:main}
\begin{cases}
-\Delta_{x} u-\Delta_{y} u + \omega u= u|u|^{{\alpha}} \ \ \ \text{in } \R_{x}^{N} \times M_{y}^{k}\\
\int_{M_{y}^{k}}\int_{\R_{x}^{N}} |u|^{2}\,dx  \dvol= \rho^{2}.
\end{cases}
\end{equation}

A special class of solutions of \eqref{eq:main} are those for which 
\[
\nabla_y u\equiv0\qquad\text{ in }\R^{N}\times M^{k},
\]
i.e. such that $u(x,y)=z(x)$, for some $z\in H^{1}(\R^{N})$. Conversely, functions in $H^{1}(\R^{N})$ can be extended to the 
product space taking them constant in $y$. In this way, thanks to the normalization choice \eqref{eq:vol}, the 
embedding $H^{1}(\R^{N})\hookrightarrow H^{1}(\R^{N} \times M^{k})$ is actually an isometry, and in the sequel we 
will identify the former space as a subspace of the latter one without further notice. The problem on 
$\R^N$ is much well understood: introducing the corresponding energy functional 
$E_0=\left.E\right|_{H^{1}(\R^{N})}$,
\begin{equation}\label{fun:main_x}
E_{0}(z)= \int_{\R_{x}^{N}}\left(\frac{1}{2}|\nabla_{x}z|^{2}-\frac{1}{2+\alpha}|z|^{2+\alpha}\right)\,dx ,
\end{equation}
one is led to consider the minimization problem
\begin{equation}\label{eq:GS_RN_subc}
I_\rho:=\inf\{E_0(z): z \in S_\rho\cap H^1(\R^N)\},
\end{equation}
which is uniquely achieved (up to translations and phase shifts), for every $\rho>0$, even for nonlinearities 
with larger exponent (up to $\alpha<\frac{4}{N}$); see Section \ref{sec:prelim} ahead for more details on the 
solutions of problem \eqref{eq:GS_RN_subc}.

One of the main questions addressed by the authors in \cite{MR3219500}, after establishing 
that \eqref{eq:GS_prod_subc} is achieved for any $\rho>0$, is whether the minimizers of \eqref{eq:GS_prod_subc} 
coincide with those of \eqref{eq:GS_RN_subc}, or if they have a nontrivial dependence on the variable $y$. 
We summarize in the next statement the part of their results that is relevant for our discussion here.
\begin{theorem}[{\cite[Thms. 1.1, 1.3]{MR3219500}}]\label{thm:TTV_main}
Let $(M^k,g)$ be a compact Riemannian manifold of dimension $k\ge1$, $N\ge1$ and
\[
0<\alpha<\frac{4}{N+k}.
\]
Then, for $K_\rho$, $I_\rho$ defined as in \eqref{eq:GS_prod_subc}, \eqref{eq:GS_RN_subc}:
\begin{enumerate}
\item\label{thm:TTV_main_1} For every $ \rho>0$, $K_\rho$ is finite, and achieved by a family of solutions to 
\eqref{eq:main}.
\item\label{thm:TTV_main_2} There exists $\rho^*\in(0,\infty)$ such that, if $\rho<\rho^*$, then every minimizer $u$ associated to  
\eqref{eq:GS_prod_subc} satisfies $\nabla u_{y}=0$; in particular, $K_\rho=I_\rho$ and the families of minimizers of 
\eqref{eq:GS_prod_subc} and \eqref{eq:GS_RN_subc} coincide.
\item\label{thm:TTV_main_3} Conversely, if $\rho>\rho^*$, then every minimizer in  
\eqref{eq:GS_prod_subc} is nontrivial in $y$ and $K_\rho<I_\rho$.
\end{enumerate}
\end{theorem}

It is worth mentioning that \cite{MR3219500} also provides the conditional orbital stability of the family of ground 
states described above, under the condition of global well-posedness for the Cauchy problem \eqref{eq:evoNLS} for 
every initial datum, as well as it discusses such well-posedness, for some specific choice of $k$ and $N$. 

Moreover, the authors suggest that the family achieving $I_\rho$ should be stable under \eqref{eq:evoNLS} 
also for $\frac{4}{N+k}\le \alpha < \frac{4}{N}$ and $\rho$ small enough; on the other hand, in the mass supercritical range, $K_\rho=-\infty$, thus they leave as an open question to 
obtain a variational description of such stable family (as they did for the case $\alpha < \frac{4}{N+k}$, 
in terms of $K_\rho$).

The main purpose of our contribution is to complement the analysis performed in  \cite{MR3219500}, in two main 
directions. First we answer the open question mentioned above, providing a characterization of the family achieving 
$I_\rho$, for $\rho$ small, as local minimizers of the energy $E$ also in the mass critical and supercritical 
cases (although still Sobolev subcritical). Secondly, we show that in some cases such local minimization problem 
undergoes a phenomenon similar to the one depicted in Theorem \ref{thm:TTV_main}, parts 
\ref{thm:TTV_main_2} and \ref{thm:TTV_main_3}, for the ground states family: namely we show that, for $\rho$ 
larger than a suitable threshold, the local minimizer may still exist and, in such a case, depend on $y$ in a nontrivial way.

More precisely, for $t^*>0$ fixed, let us consider the following problem:
\begin{equation}\label{energy_u}
m_{\rho}:=\inf\{E(u): \ u \in S_{\rho}, \ \|\nabla_{x,y} u\|^{2}_{L^{2}(\R^{N}\times M^{k})} < t^{*} \rho^{2}\}.
\end{equation}
Our main results are the following.
\begin{theorem}\label{thm:main}
Let $(M^k,g)$ be a compact Riemannian manifold of dimension $k\ge1$, $N\ge1$ and
\begin{equation}\label{eq:ass_on_alpha}
\frac{4}{N+k} \leq \alpha < \min\left\{\frac{4}{N},\frac{4}{N+k-2}\right\}.
\end{equation}
Then, for $m_\rho$, $I_\rho$ defined as in \eqref{energy_u}, \eqref{eq:GS_RN_subc}:
\begin{enumerate}
\item \label{thm:main_1} There exist $t^*>0$, independent of $\rho$, and $\thresh{ex}\in(0,\infty)$ such that, if $\rho<\thresh{ex}$, then $m_\rho$ is achieved by a family of solutions to \eqref{eq:main}.
\item \label{thm:main_2} There exists $\thresh{tr}\in(0,\thresh{ex}]$ such that, if $\rho<\thresh{tr}$, 
then every local minimizer $u$ associated to  
\eqref{energy_u} satisfies $\nabla u_{y}=0$ and is a global minimizer of \eqref{eq:GS_RN_subc}; in particular, $m_\rho=I_\rho$ and the families of
minimizers coincide.
\item \label{thm:main_3} If $  \thresh{tr} < \thresh{ex}$ then, for every 
$\rho \in ( \thresh{tr}, \thresh{ex})$, every local minimizer is nontrivial with respect to $y$ and $m_\rho<I_\rho$.
\end{enumerate}
\end{theorem}
\begin{remark}\label{rmk:sharp}
It is worth noticing that both the existence threshold $\thresh{ex}$ and the (non)triviality one 
$\thresh{tr}$ depend on the exponent $\alpha$ (and, of course, on $(M^k,g)$ and $N$).

Moreover, although we propose some more or less explicit choices of the existence threshold 
$\thresh{ex}$,  
it is not clear how one may define it to obtain a sharp characterization (i.e. non-existence of local 
minimizers for $\rho>\thresh{ex}$). Actually, its existence is based on the Gagliardo-Nirenberg 
inequality \eqref{eq:GN}, and its value is strongly influenced by the choice of $t^*$ (see Remark 
\ref{rem:ex_thresh} ahead for further comments on this issue). On the other hand, once $\thresh{tr}
<\thresh{ex}$, then the (non)triviality threshold $\thresh{tr}$ is sharp, as we show in Section 
\ref{sec:proof:thm1_2}, although not explicit (see Corollary \ref{coro:triv}). This issues have effect also when one tries to 
estimate the two thresholds and their difference, see Remark \ref{rmk:2examples} below. 
\end{remark}
\begin{remark}\label{rmk:mass_crit}
In particular our results hold true in the mass critical case 
\[
\alpha=\frac{4}{N+k}.
\] 
Actually, in this case one can combine our proof with the arguments in the 
appendix of \cite{MR3219500} to show that the local minimizers that we find are 
global ones, and one can formally choose $t^*=+\infty$ in such case.
\end{remark}
\begin{remark}\label{rmk:positive}
By using the diamagnetic inequality, we deduce that (up to a remodulation factor) 
every local minimizer is real valued. Moreover, being both $E$ and the constraint 
in the definition of $m_\rho$ even, in a standard way we can assume that every local minimizer 
is strictly positive in $\R^{N}\times M^{k}$. Throughout the paper we will always understand this 
choice,  without further notice.
\end{remark}
\begin{remark}\label{rmk:atthethreshold}
It is a natural question, in case $\thresh{tr}<\thresh{ex}$, to wonder what happens 
for $\rho=\thresh{tr}$. In this case one can easily see that $m_\rho=I_\rho$, and it admits 
minimizers constant in the variable $y$. On the other hand, nontrivial in $y$ minimizers 
may exist or not, depending on the associated bifurcation diagram. The recent paper  
\cite{dovetta2025nonuniquenessnormalizednlsground} analyzes a similar question for global minimizers 
of the NLS energy on bounded domains of $\R^N$, with Neumann boundary conditions. 
The results therein suggest that the answer to such question may depend on the exponent $\alpha$, and 
that in some cases trivial and nontrivial minimizers may coexist at $\rho=\thresh{tr}$.
\end{remark}
\begin{remark}\label{rmk:locmin}
After \cite{ntvAnPDE,MR3638314}, it has been frequently observed that in the normalized setting, 
in absence of global minimizers, local ones may appear. This has been extended to several different 
contexts, including metric graphs  \cite{MR4241295} and Mean Field Games systems \cite{CCV23}.
\end{remark}
\begin{remark}\label{rmk:sobcrit}
Notice that, as long as $k\ge2$, our results hold true for every $\alpha$ strictly smaller than 
the Sobolev critical exponent. 
Recent papers \cite{MR4096725,MR4847285,verzini2025normalizedsolutionsnonlinearschrodinger} suggest that also
in the Sobolev critical case 
\[
\alpha=\frac{4}{N+k-2}
\]
local minimizers should exist, below an appropriate threshold $\thresh{ex}$. On the other hand, in such 
case the analysis of their (non)triviality in $y$ 
may be very delicate.
\end{remark}
\begin{remark}\label{rmk:orb_stab}
Following \cite{MR3918087} it should be possible to show that the families of local minimizers 
are conditionally orbitally stable, where the conditions are that a) the (weak) solution of \eqref{eq:evoNLS} 
exists locally in time, for a time interval which is uniform in the $H^1$-norm of the initial datum and 
b) that energy and mass are conserved across the evolution.
\end{remark}
\begin{remark}\label{rmk:MP}
As a matter of fact, in the proof of Theorem~\ref{thm:main}, part \ref{thm:main_1} we show that 
$E$ has a mountain pass structure on $S_\rho$, for $\rho<\thresh{ex}$. 
It remains an open problem to find (unstable) solutions of mountain pass type.
\end{remark}
\begin{remark}\label{rmk:2examples}
Probably the most natural question after Theorem \ref{thm:main} is how to detect if 
$\thresh{tr}<\thresh{ex}$, and if such case happens under some conditions. As we mentioned, 
although sharp explicit values for the thresholds are out of reach, we provide different 
estimates in this direction. In particular, a sufficient condition for the existence of 
nontrivial local minimizers is provided in Corollary \ref{coro:criterion}, inequality 
\eqref{eq:criterion}. This is not easy to be checked, though, because it depends on 
$I_\rho$, on the constants appearing in the Gagliardo-Nirenberg inequality, and on other spectral quantities related to $M^k$, and such values are known and explicit only in very particular cases.
\end{remark}

In view of the last remark, we devote the last part of the paper to some examples in which 
we can prove that 
\begin{equation}\label{eq:strict_thresh_ineq}
\thresh{tr}<\thresh{ex},
\end{equation}
providing the existence of nontrivial local minimizers. A first general example is contained in the next 
proposition, which says that it is always possible to satisfy \eqref{eq:strict_thresh_ineq} by changing 
the metric of $M^k$, without modifying the topology, at least when the power $\alpha$ is not too large.
\begin{proposition}\label{propo:case_1}
Assume that $N\ge1$ and $\dim M^{k} =k\geq 2$. There exists a metric $\hat g$ on $M^{k}$ and $\delta>0$ such that, for every 
\[
\dfrac{4}{N+k} \leq \alpha \leq \dfrac{4}{N+k}+\delta, 
\]
\eqref{eq:strict_thresh_ineq} holds true.
\end{proposition}
\begin{remark}\label{rmk:torus}
The previous result is based on the fact that, when $k\ge2$, it is always possible to find a metric 
on $M^k$ for which the first nontrivial eigenvalue $\mu_1$ of the Laplace-Beltrami operator $-\Delta_y$ 
on $M^k$ is arbitrarily small. In the same spirit, an analogous result holds true on the flat $k$-
dimensional torus, as long as it has one frequency sufficiently small. Recently, the study of 
NLS on waveguide manifolds $\R^N\times\mathbb{T}^k$ has attracted much attention, see e.g. \cite{MR4629760,MR4784921}.
\end{remark}
To conclude, we deal with one of the few cases in which one can do ``explicit'' computations (in terms of 
special functions), that is $\R\times\sph^{k}$ where $\sph^k$ is the unitary round sphere in $\R^{k+1}$.
\begin{proposition}\label{prop:sphere}
Let $N=1$ and $M^k=\sph^k$, with the standard round metric. If
\[
\begin{split}
\text{either } &k\ge4,\qquad \frac{4}{k+1} \leq \alpha < \frac{4}{k-1},\medskip \\
\text{or } &k=3,\qquad 1 \leq \alpha < 1 + \delta,
\end{split}
\]
for a suitable $0<\delta<1$, then \eqref{eq:strict_thresh_ineq} holds true.
\end{proposition}
\begin{remark}\label{rmk:calconsph}
The previous proposition is a consequence of the fact that the constants in our estimates are somewhat explicit 
in the considered setting. Since such constants depend in various ways on special functions, and 
in particular on the Euler beta function, we have an analytic proof for the cases 
\[
k\ge 6,\  \frac{4}{k+1} \leq \alpha < \frac{4}{k-1},
\qquad\text{ and }\qquad
3\le k\le 5,\  \frac{4}{k+1} \leq \alpha < \frac{4}{k+1} + \delta_k. 
\]
The missing part of the proposition can  be checked case by case using any Computer Algebra System 
software with built-in libraries for symbolic computation with special functions. On the other hand, our 
estimates provide no information for the case $k=2$, not even with restrictions on $\alpha$. See the end of Section \ref{sec:sphere} for 
more details.
\end{remark}

The paper is structured as follows: after some preliminary results in Section \ref{sec:prelim}, 
Section \ref{sec:proof:thm1_1} is devoted to the proof of Theorem~\ref{thm:main}, part \ref{thm:main_1}, 
and to some estimates of $\thresh{ex}$ from below, and Section \ref{sec:proof:thm1_2} to the proof of 
Theorem~\ref{thm:main}, parts \ref{thm:main_2} and  \ref{thm:main_3}. In Section \ref{sec:proof:thm1_3} 
we develop some  estimates of $\thresh{tr}$ from above. Finally, Sections \ref{sec:smalleigen} and 
\ref{sec:sphere} contain the proofs of Propositions \ref{propo:case_1}, \ref{prop:sphere}, respectively.

\section{Preliminary results}\label{sec:prelim}

Let $(M^k,g)$ be a smooth compact Riemannian manifold, with $k\ge1$, and $N\ge1$. To start with, we 
recall that the Gagliardo-Nirenberg inequality can be extended naturally to product spaces $\R^{N}\times M^{k}$ (that we always endow with the natural metric). In particular, for every 
$0\le\alpha\le \frac{4}{N+k-2}$ (or $\alpha\ge0$, in case $N=k=1$) there exist positive constants $A,B$ 
such that
\begin{equation}\label{eq:GN}
\|u\|^{2+\alpha}_{L^{2+\alpha}(\R^{N}\times M^{k})} \leq A \left(\|\nabla_{x,y} u\|_{L^{2}(\R^{N}\times M^{k})}^{2}+B\|u\|_{L^{2}(\R^{N}\times M^{k})}^{2}\right)^{\theta(\alpha)/2}\|u\|^{2+\alpha-\theta(\alpha)}_{L^{2}(\R^{N}\times M^{k})},
\end{equation}
for every $u\in H^1(\R^{N}\times M^{k})$, where 
\begin{equation}\label{eq:teta}
\theta(\alpha)=\frac{(N+k)\alpha}{2}.
\end{equation}
\begin{remark}\label{rmk:A_independent}
In principle, $A$ and $B$ may depend on $(M^k,g)$, $N$, and $\alpha$, especially in case one looks for the 
best possible constants. On the other hand, in case $N+k\ge3$ (in particular, if $k\ge2$), 
a loose version of \eqref{eq:GN} can be obtained from the Sobolev inequality combined with H\"older's one. 
More precisely, since 
$M^k$ is compact and $\R^N$ is flat, both the Ricci curvature and the injectivity radius of 
$\R^{N}\times M^{k}$ are uniformly bounded from below, in an appropriate way, and we are in a position to apply 
\cite[Thm. 7.1, p. 183]{MR1688256} (just notice that here we use slightly different notations). 
As a consequence, while $B$ may still depend on all the parameters excluded $\alpha$ 
(and in particular on $g$), $A$ is explicitly determined and it only depends on the cumulative dimension  
$N+k$ (via the standard critical Sobolev constant in $H^1(\R^{N+k})$ and the bounds on 
$\theta(\alpha)$).
\end{remark}

Next we briefly review some classical results in the Euclidean space $\R^{N}$, $N \geq 1$, that will be used below. It is well known that, in the mass subcritical case 
\[
2<\alpha<\frac{4}{N},
\] 
for every $\rho>0$ there exists a solution $Z_{\rho}\in H^{1}(\R^{N})$ (depending on $\alpha$), unique up to translations, for the problem
\begin{equation}\label{eq:Zr}
\begin{cases}
-\Delta_{x} Z_{\rho}+\omega_{\rho} Z_{\rho}= Z_{\rho}^{1+\alpha} \quad \text{in } \R_{x}^{N}\\
\|Z_{\rho}\|_{L^{2}(\R^{N})}=\rho, \ \ Z_{\rho}>0,
\end{cases}
\end{equation}
where $\omega_{\rho}>0$ is uniquely determined.  Therefore $Z_\rho$ is also a solution in \eqref{eq:main}, with $
\nabla_{y} Z_{\rho}=0$. Moreover (see e.g. \cite[Sec. 2]{MR4443784}), recalling the definition of $E_0$ in \eqref{fun:main_x}, we have that $Z_\rho$ achieves 
\eqref{eq:GS_RN_subc}, namely
\begin{equation}\label{eq:Irho}
E_{0}(Z_{\rho}) = I_{\rho} = \inf\{E_0(z): z \in S_\rho\cap H^1(\R^N)\}<0.
\end{equation}
By scaling, $Z_{\rho}$ can be written in terms of the unique positive solution $U \in H^{1}(\R^{N})$ of
\begin{equation}\label{eq:U}
\begin{cases}
-\Delta_{x} U+U= U^{1+\alpha} \quad \text{in } \R_{x}^{N}\\
U>0, \ \ U(0)=\|U\|_{\infty}.
\end{cases}
\end{equation}
Indeed, setting $\|U\|_{L^{2}(\R^{N})}=\rho_{0}$, we have that
\[
Z_{1}(x)=\rho_0^{-\frac{4}{4-\alpha N}}U\left(\rho_0^{-\frac{2\alpha}{4-\alpha N}}x\right) 
\qquad\text{and}\qquad
Z_{\rho}(x)=\rho^{\frac{4}{4-\alpha N}}Z_1\left(\rho^{\frac{2\alpha}{4-\alpha N}}x\right) 
\]
 and
\begin{equation}\label{eq:nabla_Z}
\|\nabla_{x} Z_{\rho}\|^{2}_{L^{2}(\R^{N})}= -\frac{2\alpha N}{4-\alpha N} I_{\rho}, \qquad  \|Z_{\rho}\|^{2+\alpha}_{L^{2+\alpha}(\R^{N})} =- \frac{4(\alpha+2)}{4-\alpha N} I_{\rho},
\end{equation}
so that
\begin{equation}\label{eq:eneg_Z}
I_{\rho}=E_0(Z_{\rho})=- \frac{4 - \alpha N}{8-2\alpha N+4\alpha}\left(\frac{\rho_{0}}{\rho}\right)^{\frac{4\alpha}{N\alpha -4}} \rho^{2}=:-G \rho^{2+\frac{4\alpha}{4-N\alpha }},
\end{equation}
where $G=-I_1>0$ and $\rho_0$ depend on $\alpha$.

\section{Proof of Theorem~\ref{thm:main}, part \ref{thm:main_1}, and estimate of 
\texorpdfstring{$\thresh{ex}$}{rho*-ex}}
\label{sec:proof:thm1_1}

In this section, under condition \eqref{eq:ass_on_alpha} we  show that one can suitably choose $t^*$ 
in the definition \eqref{energy_u} of $m_\rho$, i.e. 
\begin{equation*}
m_{\rho}:=\inf\{E(u): \ u \in S_{\rho}, \ \|\nabla_{x,y} u\|^{2}_{L^{2}(\R^{N}\times M^{k})} < t^{*} \rho^{2}\},
\end{equation*}
and the existence threshold $\thresh{ex}>0$, in such a way that a local minimizer of $E$ on $S_{\rho}$ 
exists, whenever $0< \rho< \thresh{ex}$. With respect to \cite{MR3219500}, where the nonlinearity is 
mass-subcritical and the corresponding solutions are global minimizers, in our case $E$ is unbounded 
below on $S_\rho$, for every $\rho>0$. Then, in order to obtain a bounded minimizing sequence not 
escaping to the boundary of the subset of $S_\rho$ which we are considering, it is necessary to analyze 
the (mountain pass) geometric structure of the functional. This is the main difference with respect to 
\cite{MR3219500}.

We start by showing that $m_\rho$ is negative, independently of the choice of $t^*$.
\begin{lemma}\label{lem:negative_m_rho}
For every $\rho>0$, $t^*>0$, we have that $m_\rho<0$.
\end{lemma}
\begin{proof}
Given $Z_\rho\in S_\rho$ as in \eqref{eq:Irho}, and $h>0$, let us consider the function
\[
\tilde z_h(x) := h^{N/2} Z_\rho (hx) \in S_\rho.
\]
Then 
\[
\|\nabla_{x,y} \tilde z_h \|^{2}_{L^{2}(\R^{N}\times M^{k})} = h^2 \|\nabla_{x} Z_\rho\|^{2}_{L^{2}(\R^{N})},
\qquad
\| \tilde z_h \|^{2+\alpha}_{L^{2+\alpha}(\R^{N}\times M^{k})} = h^{\alpha N/2} \|Z_\rho\|^{2+\alpha}_{L^{2+\alpha}(\R^{N})},
\]
and choosing 
\[
h^2<\min\left\{1,\frac{t^*\rho^2}{\|\nabla_{x} Z_\rho\|^{2}_{L^{2}(\R^{N})}}\right\} = 
\min\left\{1,\frac{4-\alpha N}{2\alpha N G}\,t^*\rho^{-\frac{4\alpha}{4-\alpha N}}\right\}
\]
(where we used \eqref{eq:nabla_Z}, \eqref{eq:eneg_Z}) we obtain that $\tilde z_h $ is an admissible competitor for 
$m_\rho$.

In turn, since $h<1$ and $\frac{\alpha N}{2}<2$ (by \eqref{eq:ass_on_alpha}), this yields
\begin{equation*}
\begin{aligned}
m_\rho \le E(\tilde z_h) &=\frac{h^2}{2} \|\nabla_{x}Z_\rho \|^{2}_{L^{2}(\R^{N}) } -
\frac{h^{\alpha N/2}}{2+\alpha} \|Z_\rho\|^{2+\alpha}_{L^{2+\alpha}(\R^{N}) }\\
&\le  \frac{h^{\alpha N/2}}{2} \|\nabla_{x} Z_\rho \|^{2}_{L^{2}(\R^{N}) } -
\frac{h^{\alpha N/2}}{2+\alpha} \| Z_\rho\|^{2+\alpha}_{L^{2+\alpha}(\R^{N}) }\\
&=h^{\alpha N/2} E(Z_\rho)<0,
\end{aligned}
\end{equation*}
concluding the proof.
\end{proof}

For the purpose of investigating the geometric structure of $E$, we define the following quantity:
\begin{equation}\label{eq:scal}
t(u):=\frac{\int_{M_{y}^{k}}\int_{\R^{N}_{x}} |\nabla_{x,y}u|^{2} \,dx  \dvol}{\int_{M_{y}^{k}}\int_{\R^{N}_{x}} |u|^{2}\,dx  \dvol}.
\end{equation}
To proceed, we provide the lower estimate of $E$.
\begin{lemma}\label{lem:bdd}
There exist constants $t^{*}, \thresh{ex}>0$ such that: for all $0<\rho<\thresh{ex}$, there exists $\eps>0$ small 
such that
\begin{equation}\label{eq:starstar}
E^{*}=E^*(\rho):=\inf\{E(u):u\in S_{\rho}, \ (1-\eps)t^*\le t(u) \le t^{*}\}>0.
\end{equation}
\end{lemma}
\begin{proof}
By \eqref{eq:GN} and \eqref{eq:scal} we have, for every $u\in S_\rho$,
\begin{equation}\label{eq:def_f}
\begin{aligned}
E(u)&=\frac{1}{2} \|\nabla_{x,y}u\|^{2}_{L^{2}(\R^{N}\times M^{k}) } -\frac{1}{2+\alpha} \|u\|^{2+\alpha}_{L^{2+\alpha}(\R^{N}\times M^{k}) }\\
&\geq  \frac{1}{2} \|\nabla_{x,y}u\|^{2}_{L^{2}(\R^{N}\times M^{k}) } -\frac{A \rho^{2+\alpha-\theta(\alpha)}}{2+\alpha} \left(\|\nabla_{x,y}u\|^{2}_{L^{2}(\R^{N}\times M^{k}) } +B \rho^{2}\right)^{\theta(\alpha)/2}\\
&=\rho^{2}\left[\frac{1}{2} \frac{ \|\nabla_{x,y}u\|^{2}_{L^{2}(\R^{N}\times M^{k}) }}{\rho^{2}} -\frac{A \rho^{\alpha}}{2+\alpha} \left(\frac{\|\nabla_{x,y}u\|^{2}_{L^{2}(\R^{N}\times M^{k}) } }{\rho^{2}}+B\right)^{\theta(\alpha)/2}\right]\\
&=\rho^{2} \left [ \frac{1}{2} t(u) -\frac{A \rho^{\alpha}}{2+\alpha} \left( t(u) +B \right)^{\theta(\alpha)/2}\right]=: f(t(u),\rho).
\end{aligned}
\end{equation}
Now, in case $\alpha=\frac{4}{N+k}$, then $\theta(\alpha)=2$ and
\[
f(t,\rho)>0 \qquad\iff\qquad
\left(\frac{2+\alpha}{2A\rho^\alpha}-1\right)t>B.
\]
We infer that, in this case, it is enough to choose any
\begin{equation}\label{eq:exist_mass_crit}
\thresh{ex}<\left(\frac{2+\alpha}{2A}\right)^{1/\alpha},
\end{equation}
with $t^*$ suitably large and $\eps$ small 
(actually, also the equality is admissible, if one allows $t^*$ to depend on $\rho$, or taking $t^*=+\infty$). 

On the other hand, in case $\frac{4}{N+k} < \alpha < \min\left\{\frac{4}{N},\frac{4}{N+k-2}\right\}$, then  
$\theta(\alpha)>2$ and we claim the existence of positive $t^*$, $\rho^*$, such that
\begin{equation}\label{eq:lower}
\begin{cases}\displaystyle
f(t^{*},\rho^{*})=(\rho^{*})^{2}\left[ \frac{1}{2} t^{*} -\frac{A (\rho^{*})^{\alpha}}{2+\alpha} \left( t^{*} +B \right)^{\theta(\alpha)/2}\right]= 0\medskip\\ \displaystyle
\partial_{t}f(t^{*},\rho^{*})=(\rho^{*})^{2}\left[ \frac{1}{2}-\frac{A \theta(\alpha) (\rho^{*})^{\alpha}}{2(2+\alpha)} \left( t^{*} +B \right)^{\theta(\alpha)/2-1}\right]= 0.
\end{cases}
\end{equation}
Since $f(t,\rho)/\rho^2$ is (continuous and) strictly decreasing in $\rho$, for every fixed $t$, this will yield the desired result (with $\thresh{ex}=\rho^*$). 

On the other hand, \eqref{eq:lower} can be easily solved by direct computations: from the second equation , we obtain
\begin{equation*}
(\rho^{*})^{\alpha}=\frac{2+\alpha}{A\theta(\alpha)}(t^{*}+B)^{1-\theta(\alpha)/2},
\end{equation*}
and plugging it into \eqref{eq:lower} yields
\begin{equation*}
t^{*}=\frac{4B}{(N+k)\alpha-4} \qquad \text{and} \qquad \rho^{*} =\left[ \frac{2+\alpha}{\theta(\alpha)AB^{\theta(\alpha)/2-1}} \left( \frac{\theta(\alpha)-2}{\theta(\alpha)}\right)^{\theta(\alpha)/2-1}  \right]^{1/\alpha}.
\end{equation*}
Notice that this choice of $\thresh{ex}$ agrees (continuously, in the limit) also with that in \eqref{eq:exist_mass_crit} for 
the mass critical case.
\end{proof}
\begin{corollary}\label{coro:mrhofrombelow}
Let $\thresh{ex}$, $t^*$ as in Lemma \ref{lem:negative_m_rho}. Then, for every $0<\rho<\thresh{ex}$, 
\[
m_\rho\ge -\frac{A B ^{\theta(\alpha)/2}}{2+\alpha} \rho^{2+\alpha}.
\]
\end{corollary}
\begin{proof}
By \eqref{eq:def_f} we have, for $0<\rho<\thresh{ex}$,
\[
m_\rho \ge \min_{0\le t \le t^*} f(t,\rho) = f(0,\rho).\qedhere
\]
\end{proof}

With this choice of $t^*$, we can actually improve the content of Lemma \ref{lem:negative_m_rho}.
\begin{lemma}\label{lem:more_negative_m_rho}
Let $\thresh{ex}$, $t^*$ as in Lemma \ref{lem:negative_m_rho}. Then, for every $0<\rho<\thresh{ex}$, 
we have that 
\[
t(Z_\rho)<t^*\qquad \text{ and }\qquad m_\rho\le I_\rho<0.
\]
\end{lemma}
\begin{proof}
First of all, recall that, by \eqref{eq:nabla_Z}, \eqref{eq:eneg_Z},
\[
t(Z_\rho)=\frac{2\alpha N\,G}{4-\alpha N}\, \rho^{\frac{4\alpha}{4-N\alpha }},
\]
so that $t(Z_\rho)<t^*$ for $\rho>0$ sufficiently small. Assume by contradiction that such inequality is 
not always true in $(0,\thresh{ex})$. Then, by continuity, there exists $\bar\rho\in(0,\thresh{ex})$ such that
$t(Z_{\bar\rho})=t^*$. This yields
\[
E^*(\bar\rho) \le E(Z_{\bar\rho}) <0,
\]
in contradiction with \eqref{eq:starstar}.

Once the first claim is proved, the second one follows by definition of $m_\rho$ and $I_\rho$. 
\end{proof}

Based on the previous lemmas, and with this choice of $t^*$, $\thresh{ex}$, we have that, for every $0<\rho<\thresh{ex}$,
\[
m_{\rho}\le I_\rho< 0 < E^{*},
\] 
and therefore we can find a minimizing sequence  $(u_{n})_{n}$ for $m_{\rho}$, which is bounded by construction. 
Moreover, if we can show strong convergence (up to subsequences), then the limit $u$ satisfies 
\[
E(u)=m_\rho,\qquad\|\nabla_{x,y} u\|^{2}_{L^{2}(\R^{N}\times M^{k})} \le (1-\eps)t^{*} \rho^{2}.
\]
Thus $u$ is a local minimizer of $E$ solving \eqref{eq:main}, for a suitable Lagrange multiplier 
$\omega$, and the proof of Theorem~\ref{thm:main}, part \ref{thm:main_1}, is completed.

To prove the strong convergence of the minimizing sequence we can proceed as in the Appendix of \cite{MR3219500}, 
with minor changes. In particular, referring to the steps there, we have that all of them can be easily rephrased  
under the only assumption that the exponent $\alpha$ is Sobolev subcritical in dimension $N+k$, as we also assume, 
with three exceptions: the first step, i.e. the boundedness of the minimizing sequence, which follows here by 
construction; the third step, which gives the positive bound from below  
\begin{equation}\label{eq:bound_d}
\| u_n\|_{L^{2+\alpha}(\R^{N}\times M^{k})}\ge d>0,
\end{equation}
which also here follows by the fact that $m_\rho$ is negative (Lemma \ref{lem:more_negative_m_rho}); 
finally, the fourth step, which excludes the vanishing of the minimizing sequence (up to a sequence of 
translations in $\R^N$) by 
a localized Gagliardo-Nirenberg inequality for  $\| u\|_{L^{2+\frac{4}{N+k}}(\R^{N}\times M^{k})}$: in 
our case, 
one can argue in a standard way by considering a decomposition $\left\{Q_i\right\}_{i\in\N}$ of $\R^N$ 
into unitary cubes, and setting 
\begin{equation}\label{eq:ourlocGN1}
l_n:=\max_{i\in\N}  \|u_n\|_{L^{2+\alpha}(Q_i\times M^{k})}.
\end{equation}
Then \eqref{eq:bound_d} (and the Sobolev embedding $L^{2+\alpha}(Q_i\times M^{k})
\hookrightarrow H^{1}(Q_i \times M^{k})$, which holds with a constant $\tilde A$ independent of $i$)  
yields
\begin{equation}\label{eq:ourlocGN2}
\begin{split}
0<d &\le \| u_n\|_{L^{2+\alpha}(\R^{N}\times M^{k})}^{2+\alpha} = \sum_{i}  \|u_n\|^{2+\alpha}_{L^{2+\alpha}(Q_i\times M^{k})} \le l_n^{\alpha}\sum_{i}  \|u_n\|^{2}_{L^{2+\alpha}(Q_i\times M^{k})}\\
&\le \tilde A l_n^{\alpha}\sum_{i}  \|u_n\|^{2}_{H^{1}(Q_i\times M^{k})} = \tilde A l_n^{\alpha} \|u_n\|^{2}_{H^{1}(\R^N\times M^{k})}\le \tilde A  (t^*+1)\rho^2 \cdot l_n^{\alpha},
\end{split}
\end{equation}
and finally $l_n\ge d'>0$ for every $n$. This provides a sequence of translations $(\tau_n)_n\subset\R^N$ such that 
$u_n(x+\tau_n,y)$ has a nontrivial weak limit, up to subsequences. Then the proof can be completed as in the Appendix of \cite{MR3219500}.

\begin{remark}\label{rmk:mrhocontinuous}
An intermediate step in the proof above, that we will use later on, is that the map $\rho\mapsto m_\rho$ 
is continuous on $(0,\thresh{ex})$. This can be easily proved applying the transformation 
\[
S_{\rho_1} \ni u \mapsto \frac{\rho_2}{\rho_1} u \in S_{\rho_2}
\] 
to minimizing sequences, see \cite[Appendix, Second step]{MR3219500} for all the details.
\end{remark}
\begin{remark}[On the definition of $\thresh{ex}$]\label{rem:ex_thresh}
For future reference, we record that the proof of Lemma \ref{lem:bdd} provides as 
a first existence threshold for the 
validity of Theorem \ref{thm:main}, part \ref{thm:main_1}, the value
\[
\thresh{ex}:=\left[ \frac{2+\alpha}{\theta(\alpha)AB^{\theta(\alpha)/2-1}} \left( \frac{\theta(\alpha)-2}{\theta(\alpha)}\right)^{\theta(\alpha)/2-1}  \right]^{1/\alpha},
\]
for every $\frac{4}{N+k} \le \alpha < \min\left\{\frac{4}{N},\frac{4}{N+k-2}\right\}$ (continuously 
extended to $\alpha=\frac{4}{N+k}$). For most part of this paper this choice is satisfactory, but we 
will need to improve it at the very end of Section \ref{sec:sphere}.

Indeed, we notice that such choice is 
not optimal, and it is clear that the theorem holds true for some larger threshold, at least when 
$\alpha>\frac{4}{N+k}$. The reason for this lies in the discrepancy between the estimate of $E^*$ in 
Lemma \ref{lem:bdd} and the one of $m_\rho$ in Lemma \ref{lem:more_negative_m_rho}. Based on that, 
let us define the function
\begin{equation}\label{eq:ftilde}
\begin{split}
\tilde f(\rho)&:= \frac{1}{\rho^2}\left(f(t^*,\rho)-I_\rho\right)\\
&=\frac{1}{2} t^{*} -\frac{A \left( t^{*} +B \right)^{\theta(\alpha)/2}}{2+\alpha}\,\rho^{\alpha} 
+ G \rho^{\frac{4\alpha}{4-N\alpha }},
\end{split}
\end{equation}
with $t^{*}=\frac{4B}{(N+k)\alpha-4}$. Notice that $\tilde f(0^+)>0$, $\tilde f(+\infty)>0$, while
\[
t(Z_{\bar\rho})=t^*
\qquad\implies\qquad
f(\bar\rho)<0.
\]
Then, arguing as in Lemma \ref{lem:more_negative_m_rho}, we have that Theorem \ref{thm:main}, 
part \ref{thm:main_1}, holds true with the new threshold (larger than the previous one)
\begin{equation}\label{eq:better_thresh}
\begin{split}
\thresh{ex} &:= \sup\left\{\bar\rho:\tilde f(\rho)>0\text{ for every }0<\rho<\bar\rho\right\}\\
&= \sup\left\{\rho:\tilde f(\rho)>0,\ \tilde f'(\rho)<0\right\}.
\end{split}
\end{equation} 

In any case, also this new threshold has no reasons to be sharp: indeed it may be improved, for instance, 
changing $t^*$ or using some anisotropic Gagliardo-Nirenberg estimate (using $\|\nabla_{x} u\|_{2}^{2}
+ \lambda \|\nabla_{y} u\|_{2}^{2}$ instead of $\|\nabla_{x,y} u\|_{2}^{2}$, with some $\lambda>0$ 
depending on $\rho$); moreover, 
in any case we do not know if 
the (quasi-)optimizers of a possibly sharp Gagliardo-Nirenberg inequality belong to our constraint or not.
\end{remark}

\section{Proof of Theorem~\ref{thm:main}, parts \ref{thm:main_2} and  \ref{thm:main_3}}
\label{sec:proof:thm1_2}

In this section, we prove that there exists $0<\thresh{tr} \leq \thresh{ex}$ such that, if 
$0<\rho<\thresh{tr}$, then the family of local minimizers of $E$ on $S_{\rho}$ achieving $m_\rho$ 
consists in all the translations and phase shifts of $Z_{\rho}$. 
To this aim, it suffices to prove that, for $\rho$ small, 
any local minimizer $u$ of $E$ on $S_{\rho}$, provided by the previous section, satisfies 
$\nabla_{y} u\equiv0$, and consequently $m_\rho= I_\rho$. Indeed, in such a case, $u$ would be a solution, positive by Remark \ref{rmk:positive}, of \eqref{eq:Zr}. Moreover, it will be clear by construction that,  
in case $\rho>\thresh{tr}$, then $m_\rho< I_\rho$.

To begin with, we follow once again the strategy in \cite{MR3219500} and we apply the transformation
\begin{equation}\label{eq:1torhowithlambda}
u=\rho^{4/(4-\alpha N)}v(\rho^{2\alpha/(4-\alpha N)}x,y),
\end{equation}
in such a way that $u \in S_{\rho}$ if and only if $v \in S_{1}$. Thus,
\begin{equation*}
\begin{aligned}
E(u)=&E(\rho^{4/(4-\alpha N)}v(\rho^{2\alpha}/(4-\alpha N)x,y))\\
=&\rho^{(8-2\alpha N+4\alpha)/(4-\alpha N)}\left(\frac{1}{2}\rho^{-4\alpha/(4-\alpha N)}\int_{M_{y}^{k}}\int_{\R_{x}^{N}}|\nabla_{y} v|^{2}\,dx  \dvol\right.\\
&\left.+\frac{1}{2}\int_{M_{y}^{k}}\int_{\R_{x}^{N}}|\nabla_{x} v|^{2}-\frac{1}{2+\alpha}|v|^{2+\alpha}\,dx \dvol\right).
\end{aligned}
\end{equation*}
Letting 
\[
\lambda =\rho^{-4\alpha/(4-\alpha N)},
\] 
we define
\begin{equation}\label{fun:main_lambda}
E_{\lambda}(u)= \int_{M_{y}^{k}}\int_{\R_{x}^{N}}\left(\frac{1}{2}|\nabla_{x}u|^{2} 
+ \frac{\lambda}{2}|\nabla_{y}u|^{2}-\frac{1}{2+\alpha}|u|^{2+\alpha}\right)\,dx  \dvol,
\end{equation}
so that, under \eqref{eq:1torhowithlambda}, 
\begin{equation}\label{eq:EvsElambda}
E(u) = \rho^{2+\frac{4\alpha}{4-N\alpha}}E_\lambda(v).
\end{equation}
Then it is clear that $u$ is a local minimizer of $E$ on $S_{\rho}$ achieving $m_\rho$ if and only if the 
corresponding $v$, as defined in \eqref{eq:1torhowithlambda}, is a local minimizer of the functional $E_{\lambda}$ on $S_{1}$ associated to the minimization problem
\begin{equation}\label{energy_au}
m_{\lambda}:=\inf\left\{E_{\lambda}: \ u \in S_{1},\ \frac{1}{\lambda}\|\nabla_x u\|^{2}_{L^{2}(\R^{N}\times M^{k})} + 
\|\nabla_y u\|^{2}_{L^{2}(\R^{N}\times M^{k})} < t^{*} \right\}.
\end{equation}
In turn, $m_\lambda$ is achieved for every $\lambda>\lambda^*_{\mathrm{ex}}:= (\thresh{ex})^{-4\alpha/(4-\alpha N)}$, by the 
previous section.

We notice that, as in the mass subcritical 
case, this change of variable makes apparent in the energy functional the role of $\nabla_y u$:
indeed, for every fixed $u\in S_1$, it is clear that
\[
\lambda_1<\lambda_2
\qquad\implies \qquad
E_{\lambda_1}(u)\le E_{\lambda_2}(u),
\]
with equality if and only if  $u\in S_1\cap H^1(\R^N)$. On the other hand, in the present case, 
also the constraint involving $t^*$ is affected by a change of $\lambda$, and
\[
\lambda_1<\lambda_2,\quad
\frac{1}{\lambda_2}\|\nabla_x u\|^{2}_{2} + 
\|\nabla_y u\|^{2}_{2} < t^{*}
\qquad\rlap{$\quad\not$}\implies \qquad
\frac{1}{\lambda_1}\|\nabla_x u\|^{2}_{2} + 
\|\nabla_y u\|^{2}_{2} < t^{*}.
\]
This yields some additional difficulties, that we overcome in the next lemma.
\begin{lemma}\label{lem:topointer}
There exists $\thresh{tr}\in[0,\thresh{ex}]$ such that, defining
\[
\begin{split}
D_0&:=\left\{0<\rho<\thresh{ex}:m_\rho<I_\rho\right\}, \\
D_1&:=\left\{0<\rho<\thresh{ex}:m_\rho\text{ is achieved by some $u$ with }\nabla_y u\not\equiv0\right\},
\end{split}
\]
then
\[
D_0=(\thresh{tr},\thresh{ex})\qquad\text{and}\qquad D_1=
\begin{cases}
\text{either }&(\thresh{tr},\thresh{ex}),\\
\text{or }&[\thresh{tr},\thresh{ex})
\end{cases}
\]
(understanding $D_0=D_1=\emptyset$ if $\thresh{tr}=\thresh{ex}$).
\end{lemma}
\begin{proof}
Notice that, by Remark \ref{rmk:mrhocontinuous}, $D_0$ is an open set, while by \eqref{eq:Irho} 
we have that $D_0\subset D_1$. Consequently, if $D_1$ is empty there is nothing to prove (taking $\thresh{tr}=\thresh{ex}$). In the opposite case, we claim that
\begin{equation}\label{eq:rho2lambda}
\bar\rho \in D_1\qquad\implies\qquad (\bar \rho,\thresh{ex}) \subset D_0.
\end{equation}
This yields the conclusion, since it implies $D_0=  (\inf(D_1),\thresh{ex})\subset D_1\subset (0,\thresh{ex})$.

We prove the claim in the setting of $E_\lambda$, $m_\lambda$. Notice that the transformation 
\eqref{eq:1torhowithlambda} associates $u=Z_\rho$ with $v=Z_1$, which in turn is always a competitor in the definition 
of $m_\lambda$ by Lemma \ref{lem:more_negative_m_rho}. Under this setting, we obtain the following version of claim 
\eqref{eq:rho2lambda}:
\begin{equation*}
\begin{cases}
\exists \bar \lambda > \lambda^*_{\mathrm{ex}},\ \bar u \in S_1:\medskip\\
E_{\bar\lambda}(\bar u) = m_{\bar\lambda} \le E_{\bar\lambda}(Z_1)<0,\smallskip\\
\frac{1}{\bar\lambda}\|\nabla_x \bar u\|^{2}_{2} + 
\|\nabla_y \bar u\|^{2}_{2} < t^{*}, \smallskip\\
\|\nabla_y \bar u\|^{2}_{2}>0,
\end{cases}
\qquad\implies\qquad 
\forall \lambda\in(\lambda^*_{\mathrm{ex}},\bar\lambda),\ m_\lambda < E_{\lambda}(Z_1).
\end{equation*}

Hence, let $\bar\lambda,\bar u$ as above. Then, for every $\lambda^*_{\mathrm{ex}}<\lambda <\bar \lambda$, 
we infer 
\[
\frac{1}{\lambda}\|\nabla_x \bar u\|^{2}_{2} + 
\|\nabla_y \bar u\|^{2}_{2} <  t^*;
\]
indeed, arguing as in the proof of Lemma \ref{lem:more_negative_m_rho}, if this was not the case, we would find 
$\lambda^*_{\mathrm{ex}}<\lambda' <\bar \lambda$ such that 
\[
\frac{1}{\lambda'}\|\nabla_x \bar u\|^{2}_{2} + 
\|\nabla_y \bar u\|^{2}_{2} = t^*
\qquad\text{ and }\qquad
E_{\lambda'}(\bar u) < E_{\bar\lambda}(\bar u)<0,
\]
in contradiction with Lemma \ref{lem:bdd} and \eqref{eq:1torhowithlambda} (notice that $E$ and $E_\lambda$ have the 
same sign on corresponding functions). 

Then $\bar u$ is an admissible competitor for every $\lambda^*_{\mathrm{ex}}<\lambda' <\bar \lambda$, providing
\[
m_\lambda\le E_\lambda(\bar u) < E_{\bar \lambda} (\bar u) \le E_{\bar \lambda} (Z_1) = E_\lambda(Z_1).
\]
This strict inequality gives the claim, which in turn provides the lemma.
\end{proof}

\begin{corollary}[Definition of $\thresh{tr}$]\label{coro:triv}
Let us define 
\[
\thresh{tr} := \inf\left\{0<\rho<\thresh{ex}:m_\rho<I_\rho\right\}
\]
(meaning that $\thresh{tr}=\thresh{ex}$ in case such set is empty). Then $\thresh{tr}\in[0,\thresh{ex}]$ satisfies:
\begin{enumerate}
\item $0<\rho<\thresh{tr}$ implies that every minimizer associated to $m_\rho$ is constant in the variable  $y$, and 
$m_\rho=I_\rho$;
\item $\thresh{tr}<\rho<\thresh{ex}$ implies that that every minimizer associated to $m_\rho$ is nontrivial in the variable  $y$, and $m_\rho<I_\rho$.
\end{enumerate}
\end{corollary}

Based on the previous corollary, to conclude the proof of Theorem \ref{thm:main}
we are only left to prove that $\thresh{tr}>0$. This is a long procedure, but it can be done following 
step by step the argument in \cite[Section 3]{MR3219500}: since $\lambda =\rho^{-4\alpha/(4-\alpha N)}$, 
we choose a sequence $\lambda_{j} \to \infty$ (i.e. $\rho_j \to 0$) and a corresponding sequence of local minimizers $(u_{\lambda_{j}})_j$ with
\[
\frac{1}{\lambda_j}\|\nabla_x u_{\lambda_{j}}\|^{2}_{L^{2}(\R^{N}\times M^{k})} + 
\|\nabla_y u_{\lambda_{j}}\|^{2}_{L^{2}(\R^{N}\times M^{k})} < t^{*},
\qquad
E_{\lambda_{j}}(u_{\lambda_{j}}) =m_{\lambda_{j}}\le E_{\lambda_{j}}(Z_1)= I_1,
\]
independent of $j$. Moreover, we assume without loss of generality that
\begin{equation*}
u_{\lambda_{j}}(x,y) > 0, \quad \mbox{for all } (x,y)\in \R^{N}\times M^{k}.
\end{equation*}
Following \cite{MR3219500} our final goal is to show that $\nabla u_{\lambda_{j}}\equiv 0$ for $j$ sufficiently 
large. A preliminary, key step consists in showing that $\|\nabla_y u_{\lambda_{j}}\|_{L^{2}(\R^{N}\times M^{k})}\to 0$ as 
$j\to\infty$. In the mass subcritical case, this last fact is based on a global bound from below for $E_\lambda$, 
uniform in $\lambda$ large. Since we are in a mass supercritical case, this bound is an issue, as suggested also 
by Corollary \ref{coro:mrhofrombelow}, so we have to resort to a refined use of the 
Gagliardo-Nirenberg inequality \eqref{eq:GN}.
\begin{lemma}\label{lem:mlambdafrombelow}
\begin{equation}\label{eq:to_infy_SNN}
\lim_{j \to \infty} \|\nabla_y u_{\lambda_{j}}\|_{L^{2}(\R^{N}\times M^{k})}=0.
\end{equation}
\end{lemma}
\begin{proof}
We plug the change of variable \eqref{eq:1torhowithlambda} into \eqref{eq:GN} to obtain, for every $v\in S_1$ and for every $\rho,\lambda$,
\begin{equation*}
\rho^{2+\frac{4\alpha}{4-\alpha N}}\|v\|^{2+\alpha}_{2+\alpha} \leq A \left(\rho^{2+\frac{4\alpha}{4-\alpha N}}
\|\nabla_x v\|_{2}^{2} + \rho^2\|\nabla_y v\|_{2}^{2} +B\rho^2\right)^{\theta/2}\rho^{2+\alpha-\theta},
\end{equation*}
which yields, using $\alpha - \frac{4\alpha}{4-\alpha N} = - \frac{\alpha^2N}{4-\alpha N} = - \frac{\alpha N}{4}
\cdot \frac{4\alpha}{4-\alpha N}$, 
\begin{equation*}
\|v\|^{2+\alpha}_{2+\alpha} \leq \lambda^{\alpha N/4} A \left(\frac{1}{\lambda} 
\|\nabla_x v\|_{2}^{2} + \|\nabla_y v\|_{2}^{2} +B\right)^{\theta/2}.
\end{equation*}
In particular we have, for every $j$,
\begin{equation*}
\|u_{\lambda_{j}}\|^{2+\alpha}_{2+\alpha} \leq \lambda_j^{\alpha N/4} A \left(t^* +B\right)^{\theta/2}.
\end{equation*}
Assume by contradiction that (up to a subsequence)
\[
\|\nabla_y u_{\lambda_{j}}\|_{2}^2\ge 2\eps_0>0.
\]
Then
\[
\begin{split}
I_1 \ge m_{\lambda_j} \ge \frac{\lambda_j}{2}\|\nabla_y u_{\lambda_{j}}\|_{2}^2 - \frac{1}{2+\alpha}\|u_{\lambda_{j}}\|^{2+\alpha}_{2+\alpha} \ge \eps_0 \lambda_j - C \lambda_j^{\alpha N/4}
\end{split}
\]
for every $j$, which is a contradiction since $\lambda_j\to\infty$ as $j\to\infty$, and 
$\frac{\alpha N}{4}<1$.
\end{proof}
Now we are in a position to implement the strategy in \cite{MR3219500}. 
Actually, the main differences here are in the first lemma, which corresponds to 
\cite[Lemma 3.2]{MR3219500}, and can be replaced by the following one.
\begin{lemma}\label{lem:ourlemma3.2}
We have
\begin{equation}\label{eq:E_eq_I}
\lim_{j \to \infty} m_{\lambda_{j}}=I_{1}
\end{equation}
and
\begin{equation}\label{eq:lambda_to0}
\lim_{j \to \infty} \lambda_{j}\|\nabla_y u_{\lambda_{j}}\|^2_{L^{2}(\R^{N}\times M^{k})}=0.
\end{equation}
\end{lemma}
\begin{proof}
Clearly
\begin{equation}\label{eq:left}
\limsup_{j \to \infty} m_{\lambda_{j}}\leq I_{1},
\end{equation}
hence we have to prove the bound from below. Exploiting the definition of $I_\rho$ in \eqref{eq:eneg_Z}
we have
\begin{equation}\label{eq:dimred3.2}
\begin{split}
m_{\lambda_{j}}
& = \frac{\lambda_{j}}{2} \|\nabla_y u_{\lambda_{j}}\|^2_2 +  \int_{M_{y}^{k}}\left(\int_{\R_{x}^{N}}\frac{1}{2}|\nabla_{x}u_{\lambda_{j}}|^{2}\,dx -\frac{1}{2+\alpha}|u_{\lambda_{j}}|^{2+\alpha}\,dx \right)\,\dvol\\ 
&=  \frac{\lambda_{j}}{2} \|\nabla_y u_{\lambda_{j}}\|^2_2 + \int_{M_{y}^{k}} E_0(u_{\lambda_{j}}(\cdot,y))\,\dvol\\
&\ge  \frac{\lambda_{j}}{2} \|\nabla_y u_{\lambda_{j}}\|^2_2 +  I_1 \int_{M_{y}^{k}} \|u_{\lambda_{j}}(\cdot,y))\|_{L^{2}_{x}}^{2+\frac{4\alpha}{4-N\alpha }}\,\dvol.
\end{split}
\end{equation}
and we are left to show that the last integral converges to $1$. To this aim, differently from \cite[Lemma 3.2]{MR3219500}, here we introduce the functions
\[
h_{j}(y):=\|u_{\lambda_{j}}(\cdot,y)\|_{L^{2}_{x}}.
\]
We know that, for every $j$,
\begin{equation}\label{eq:SSN_L2}
\|h_{j}(y)\|_{L^{2}_{y}}^{2}=\int_{M_{y}^{k}}h^{2}_{j}\dvol=1.
\end{equation}
and
\begin{equation*}
\begin{aligned}
\int_{M_{y}^{k}}|\nabla_{y}h_{j}(y)|^{2}&\,\dvol=\int_{M_{y}^{k}}(\nabla_{y} \| u_{\lambda_{j}} (x,y)\|_{L^{2}_{x}})^{2}\dvol\\
&\leq \int_{M_{y}^{k}}\left[  \frac{1}{2}\left( \int_{\R_{x}^{N}} u_{\lambda_{j}}^{2} \,dx  \right)^{-\frac{1}{2}} \int_{\R_{x}^{N}} |u_{\lambda_{j}}||\nabla_{y} u_{\lambda_{j}}| \,dx   \right]^{2} \dvol\\
&\leq C \int_{M_{y}^{k}}\left[  \frac{1}{2}\left( \int_{\R_{x}^{N}} u_{\lambda_{j}}^{2} \,dx  \right)^{-\frac{1}{2}} \left( \int_{\R_{x}^{N}} u_{\lambda_{j}}^{2} \,dx  \right)^{\frac{1}{2}} \left(\int_{\R_{x}^{N}} |\nabla_{y} u_{\lambda_{j}}|^{2} \,dx  \right)^{\frac{1}{2}}  \right]^{2} \dvol\\
&\leq C \|\nabla_{y} u_{\lambda_{j}}\|^{2}_{L^{2}_{x,y}}.
\end{aligned}
\end{equation*}
Thus Lemma \ref{lem:mlambdafrombelow} implies
\begin{equation}\label{eq:SSN_N_L2}
\lim_{j \to \infty} \|\nabla_{y} h_{j}\|^{2}_{L^{2}_{y}}=0.
\end{equation}
As a consequence, defining the mean value of $h_j$ as
\[
\bar h_j:=\int_{M_{y}^{k}}h_{j}(y)\dvol\in(0,1],
\]
Poincaré-Wirtinger inequality yields
\[
\int_{M_{y}^{k}}|h_{j}(y) - \bar h_j|^2\dvol \le C_{PW}\int_{M_{y}^{k}}|\nabla_{y}h_{j}(y)|^{2}	\to 0.
\]
Thus, up to subsequences, both the constants $\bar h_j$ and the functions $h_j$ converge to the same 
constant, strongly in $H^1(M^k)$. Finally, such constant is $1$, because of \eqref{eq:SSN_L2} 
(and \eqref{eq:vol}).

Since $h_j\to1$ in $H^1(M^k)$ and
\[
k\ge3,\ \alpha < \frac{4}{N+k-2} 
\qquad\iff\qquad
2 + \frac{4\alpha}{4-N\alpha } < 2 + \frac{4}{k-2},
\]
the Sobolev embedding yields, for any $k\ge1$, 
\begin{equation}\label{eq:new_kk}
\lim_{j \to \infty}\| w_{j} -1\|_{L^{r}_{y}}=0 \quad \text{for } r=2+\frac{4\alpha}{4-N\alpha },
\end{equation}
and we can pursue the estimate in \eqref{eq:dimred3.2} as
\begin{equation}\label{eq:right_d}
m_{\lambda_{j}} \geq  \frac{\lambda_{j}}{2} \|\nabla_y u_{\lambda_{j}}\|^2_2 +  I_{1} \int_{M_{y}^{k}} w_{j}(y)^{2+\frac{4\alpha}{4-N\alpha }} \dvol \ge I_{1}+o(1),
\end{equation}
which, together with \eqref{eq:left}, implies both \eqref{eq:E_eq_I} and  \eqref{eq:lambda_to0}.
\end{proof}

After the previous lemma, the proof of the fact that $\nabla u_{\lambda_{j}}\equiv 0$ for $j$ 
sufficiently large (which in turn yields $\thresh{tr}>0$) can be completed as in 
\cite[Section 3]{MR3219500} with minor changes. Indeed, using Lemma \ref{lem:ourlemma3.2}, it is 
possible to pass to the limit in the equation satisfied by $u_{\lambda_{j}}$, obtaining that, 
up to subsequences, translations and modulations, $ u_{\lambda_{j}} \to Z_1$ strongly in 
$H^1(\R^{N}\times M^{k})$.
As in the proof of Theorem~\ref{thm:main}, part \ref{thm:main_1}, for such a strong convergence we 
can not use the mass-critical localized Gagliardo-Nirenberg inequality as in \cite{MR3219500}, but 
we can argue in the standard way described in \eqref{eq:ourlocGN1}, \eqref{eq:ourlocGN2}. Once we 
know that $ u_{\lambda_{j}} \to Z_1$, we can conclude as in \cite[Lemma 3.6]{MR3219500}, where the 
main assumption is that $2+\alpha$ is Sobolev subcritical (as in our case).

\section{Estimate of \texorpdfstring{$\thresh{tr}$}{rho*-tr}}\label{sec:proof:thm1_3}

In this section we provide an estimate of $\thresh{tr}$ which, combined with Remark \ref{rem:ex_thresh}, 
will allow to show that in some cases  $\thresh{tr}<\thresh{ex}$, so that Theorem~\ref{thm:main}, part 
\ref{thm:main_3} applies and $E$ admits nontrivial local minimizers on $S_{\rho}$.

Among the possible choices, we will infer such estimate by imposing that $E''(Z_\rho)$ is not positive 
semidefinite. To this aim, we introduce the first nontrivial eigenvalue $\mu_{1}>0$ of the 
Laplace-Beltrami operator $-\Delta_y$ on $M^k$, and the corresponding eigenfunction 
$\varphi_{1}=\varphi_1(y)$:
\begin{equation}\label{eq:eigenvalue}
\mu_1 = \min \left\{\frac{\|\nabla w\|^2_{L^{2}(M^{k})}}{\|w\|^2_{L^{2}(M^{k})}} : 
w\in H^{1}(M^{k}),\   \int_{M_{y}^{k}} w \dvol=0\right\},\quad - \Delta_{y} \varphi_{1} =\mu_{1} \varphi_{1}\text{ in } M_{y}^{k},
\end{equation}
with
\begin{equation}\label{eq:norm_mu1}
\int_{M_{y}^{k}} \varphi_1 \dvol=0,
\qquad\qquad 
\|\varphi_{1}\|^2_{L^{2}(M^{k})}=1. 
\qquad\text{ and }\qquad 
\|\nabla\varphi_{1}\|^2_{L^{2}(M^{k})}=\mu_1. 
\end{equation}
We prove the following.
\begin{proposition}\label{prop:rhotrfrombelow}
We have
\begin{equation}\label{eq:rho_u}
\thresh{tr}\le \left[ \frac{\mu_{1}(4-\alpha N)}{G(4(1+\alpha)(2+\alpha)-2\alpha N)} \right]^{\frac{4-N\alpha}{4\alpha}},
\end{equation}
where $G>0$ is defined as in \eqref{eq:eneg_Z}. 
\end{proposition}
\begin{proof}
We are going to show that, if $\rho$ is larger than the right hand side of \eqref{eq:rho_u}, then
\[
E^{''}(Z_{\rho})[\varphi_{1}Z_{\rho},\varphi_{1}Z_{\rho}]<0.
\] 
As a consequence, in any neighborhood of $Z_{\rho}$ 
there exists a function $u\in S_\rho$ such that
\begin{equation*}
E(u)<E(Z_{\rho}),
\end{equation*}
and the proposition follows.

By \eqref{eq:norm_mu1} we have
\begin{equation*}
\begin{aligned}
E^{''}(Z_{\rho})[\varphi_{1}Z_{\rho},\varphi_{1}Z_{\rho}]&= \int_{M_{y}^{k}}\int_{\R_{x}^{N}} 
|\nabla_{x,y} (\varphi_{1}Z_{\rho})|^{2} \,dx   \dvol -(1+\alpha) \int_{M_{y}^{k}}\int_{\R_{x}^{N}} Z_{\rho}^{2+\alpha} \varphi_{1}^{2} \,dx   \dvol \\
&=  \mu_{1} \rho^{2}  +\int_{\R_{x}^{N}} |\nabla_x Z_{\rho}|^{2}\,dx -(1+\alpha) \int_{\R_{x}^{N}} Z_{\rho}^{2+\alpha} \,dx .
\end{aligned}
\end{equation*}
From \eqref{eq:nabla_Z}, we obtain that
\begin{equation*}
E^{''}(Z_{\rho})[\varphi_{1}Z_{\rho},\varphi_{1}Z_{\rho}]= \mu_{1} \rho^{2}  +  \frac{4(1+\alpha)(2+\alpha)-2\alpha N}{4-\alpha N}I_{\rho}.
\end{equation*}
Since $\alpha<\frac{4}{N}$, we have
\[
\frac{4(1+\alpha)(2+\alpha)-2\alpha N}{4-\alpha N}=2 + \frac{4\alpha(3+\alpha)}{4-\alpha N} >0.
\]
Then, from \eqref{eq:eneg_Z}, it follows that
\begin{equation*}
\frac{4(1+\alpha)(2+\alpha)-2\alpha N}{4-\alpha N}G \rho^{\frac{4\alpha}{4-N\alpha}}>\mu_{1}
\qquad\implies\qquad
E^{''}(Z_{\rho})[\varphi_{1}Z_{\rho},\varphi_{1}Z_{\rho}]<0,
\end{equation*}
concluding the proof.
\end{proof}
\begin{remark}\label{rmk:other_eig}
Instead of $\varphi_1$ one may use any other $\varphi_k$, $k\ge2$, although getting a 
worse condition.
\end{remark}
Of course, the above proposition is informative only when the right hand side in \eqref{eq:rho_u} is smaller 
than $\thresh{ex}$. By Remark \ref{rem:ex_thresh} we infer the following criterion.
\begin{corollary}\label{coro:criterion}
If 
\begin{equation}\label{eq:criterion}
\left[ \frac{\mu_{1}(4-\alpha N)}{G(4(1+\alpha)(2+\alpha)-2\alpha N)} \right]^{\frac{4-N\alpha}{4}}
<
\frac{2+\alpha}{\theta(\alpha)AB^{\theta(\alpha)/2-1}} \left( \frac{\theta(\alpha)-2}{\theta(\alpha)}\right)^{\theta(\alpha)/2-1}
\end{equation}
then
\[
\thresh{tr}<\thresh{ex}, 
\]
and Theorem~\ref{thm:main}, part \ref{thm:main_3} applies providing nontrivial local minimizers of $E$ 
on $S_{\rho}$ (notice that the right hand side extends continuously to $\frac{2+\alpha}{2A}$ as 
$\alpha\to\frac{4}{N+k}$).
\end{corollary}
\begin{remark}\label{rem:criterionwithvolume}
For some applications, it is convenient to remove the normalization $\vol(M^k)=1$ in \eqref{eq:vol}. 
By a direct check, one can easily verify that the definition of $\thresh{ex}$ in Section 
\ref{sec:proof:thm1_1} is unchanged. On the other hand, Proposition \ref{prop:rhotrfrombelow} needs to 
be rephrased, using $Z_{\hat\rho}$ instead of  $Z_\rho$, where
\[
\hat\rho^2 = \frac{\rho^2}{\vol(M^k)}\qquad\text{so that }
 \|Z_{\hat\rho}\|_{L^{2}(\R^{N}\times M^{k})}=\rho.
\]
In this case, \eqref{eq:criterion} becomes
\begin{equation}\label{eq:criterionwithvolume}
(\vol(M^k))^\frac{\alpha}{2}\left[ \frac{\mu_{1}(4-\alpha N)}{G(4(1+\alpha)(2+\alpha)-2\alpha N)} \right]^{\frac{4-N\alpha}{4}}
<
\frac{2+\alpha}{\theta(\alpha)AB^{\theta(\alpha)/2-1}} \left( \frac{\theta(\alpha)-2}{\theta(\alpha)}\right)^{\theta(\alpha)/2-1}.
\end{equation}
\end{remark}
\begin{remark}\label{rem:improvedcriterionwithvolume}
Taking into account Remark \ref{rem:ex_thresh}, we know that Theorem \ref{thm:main} holds true also 
with the improved threshold $\thresh{ex}$ defined in \eqref{eq:better_thresh}. With this choice, we 
obtain that a sufficient condition to obtain $\thresh{tr}<\thresh{ex}$ is that
\[
\begin{cases}
\tilde f(R)>0\\ \tilde f'(R)<0,
\end{cases}
\qquad\text{ where }\quad
R=(\vol(M^k))^\frac{1}{2}\left[ \frac{\mu_{1}(4-\alpha N)}{G(4(1+\alpha)(2+\alpha)-2\alpha N)} \right]^{\frac{4-N\alpha}{4\alpha}}
\]
and $\tilde f$ is defined as in \eqref{eq:ftilde}.
\end{remark}

\section{Proof of Proposition \ref{propo:case_1}}\label{sec:smalleigen}

In this section, we use Corollary \ref{coro:criterion} to prove Proposition \ref{propo:case_1}. In 
the following, $N\ge1$ and $M^k$, with $k\ge2$, are fixed, while we use both $\alpha$ (still 
satisfying \eqref{eq:ass_on_alpha}) and the metric $g$ on $M^k$ as parameters. The starting point 
is the following well known fact.

\begin{lemma}\label{lem:mu_eps}
Let $M^{k}$ be a compact Riemannian manifold with $\dim M^k \geq 2$. Then, for any $\varepsilon > 0$, there exists a Riemannian metric $g_\varepsilon$ on $M^{k}$ such that
\[
\vol_{g_\varepsilon}(M^{k})=1, \qquad \mbox{and } \ 0<\mu_{1}(M^{k},g_\varepsilon) \leq \varepsilon,
\]
where $\mu_{1}(M^{k},g_\varepsilon)$ denotes the first nonzero eigenvalue of $-\Delta_y$ on $M^k$ 
(recall \eqref{eq:eigenvalue}).
\end{lemma}
For a proof and further information concerning this result see e.g. \cite[Example 15]{MR3748520}, or 
\cite[Ch. IV, Sec. 5]{MR768584}.

The idea is to use this result in order to take $\mu_1$ very small in such a way that 
\eqref{eq:criterion} is satisfied. Actually, this cannot be done in a direct way, since changing 
$g$ also affects $B$, in a way that we do not know how to control. Since, on the other hand, $A$ does 
not depend on $g$ (see Remark \ref{rmk:A_independent}), we can implement the above strategy for $\theta(\alpha)$ near $2$.

Precisely, let us notice that, with our choice, all the quantities in \eqref{eq:criterion} are fixed 
but:
\[
\alpha,\quad \mu_1=\mu_1(g), \quad \theta =\theta(\alpha), \quad A=A(\alpha), \quad B=B(g).
\]
In particular, the last two quantities are detected in \cite[Thm. 7.1, p. 183]{MR1688256} as 
$A=\tilde A^{\theta(\alpha)/2}$ and $B=\tilde B/\tilde A$, where $\tilde A$ is an explicit constant 
only depending on $N+k$ and $\tilde B$ also depends on $(M^k,g)$. In particular, all the above quantities 
are continuous in $\alpha$.

Let us define
\[
L(\alpha,g) = \left[ \frac{\mu_{1}(g)(4-\alpha N)}{G(4(1+\alpha)(2+\alpha)-2\alpha N)} \right]^{\frac{4-N\alpha}{4}}
\]
and
\[
R(\alpha,g) = 
\begin{cases}
\displaystyle\frac{2+\alpha}{2A(\alpha)} & \displaystyle\text{if }\alpha=\frac{4}{N+k}\\
\displaystyle\frac{2+\alpha}{\theta(\alpha)A(\alpha)B(g)^{\frac{\theta(\alpha)}{2}-1}} \left( \frac{\theta(\alpha)-2}{\theta(\alpha)}\right)^{\frac{\theta(\alpha)}{2}-1} & \displaystyle\text{if }\frac{4}{N+k} < \alpha < \min\left\{\frac{4}{N},\frac{4}{N+k-2}\right\}.
\end{cases}
\]
Then both $L$ and $R$ are continuous with respect to $\alpha$, and \eqref{eq:criterion} writes
\[
L(\alpha,g) < R(\alpha,g).
\]
Now, let us first consider $\alpha= \alpha_* :=\frac{4}{N +k}$, with 
$\theta(\alpha_*)=2$. Then $R(\alpha_*,g)$ is independent of $g$. Using Lemma \ref{lem:mu_eps} we
can fix a metric $\hat g$ such that $\mu_1(\hat g)$ is small enough to satisfy
\[
L(\alpha_*,\hat g) <  \frac{2+\alpha_*}{2A(\alpha_*)} = R(\alpha_*,\hat g).
\]
By continuity, there exists $\delta>0$ such that
\[
\alpha_* \le \alpha \le \alpha_*+\delta 
\qquad\implies\qquad
L(\alpha,\hat g) < R(\alpha,\hat g),
\]
and the theorem follows.

\section{Proof of Proposition \ref{prop:sphere}}\label{sec:sphere}

As we mentioned, the case $\R^N\times M^k = \R\times \sph^k$, $k\ge2$, allows for some explicit calculations of the 
constants $G$, $A$, $B$.

As long as $G$ is concerned, in dimension $N=1$ the solution $Z_\rho$ is explicit. 
Indeed, in this case, direct computations show that
\[
U(x)=\left(1+\frac{\alpha}{2}\right)^{\frac{1}{\alpha}}\left(\cosh \left(\frac{\alpha}{2}x\right)\right)^{-\frac{2}{\alpha}}
\]
is the unique solution of
\[
\begin{cases}
-U''+U= U^{1+\alpha} \quad \text{in } \R\\
U>0, \ \ U(0)=\|U\|_{\infty}.
\end{cases}
\]
Recalling \eqref{eq:eneg_Z}, we have
\begin{equation}\label{eq:G_special}
G=\frac{4-\alpha}{2(4+\alpha)} \rho_{0}^{-\frac{4\alpha}{4-\alpha}},
\end{equation}
where
\[
\rho_0^2 = \int_{\R}U^{2}(x)\,dx = \left(1+\frac{\alpha}{2}\right)^{\frac{2}{\alpha}}
\int_{\R}\left(\cosh \left(\frac{\alpha}{2}x\right)\right)^{-\frac{4}{\alpha}}\,dx = 
\left(1+\frac{\alpha}{2}\right)^{\frac{2}{\alpha}}\,\frac{2}{\alpha}\,\Beta\left(\frac12,\frac2\alpha\right),
\]
where the Euler beta function $\Beta$ is defined in terms of the Gamma function $\Gamma$ as
\[
\Beta(x,y)=\frac{\Gamma\left(x\right)\Gamma\left(y\right)}{\Gamma\left(x+y\right)}
\]
(see e.g. \cite[pp. 9--11]{MR58756}).

On the other hand, let $(M^k,g)=(\sph^k,h)$, $k\ge2$, be the standard round unit sphere of $\R^{k+1}$, 
with volume
\[
\vol(\sph^k)=\omega_k=\frac{2\pi^{\frac{k+1}{2}}}{\Gamma\left(\frac{k+1}{2}\right)}.
\] 
Then, setting $2^*:=2+\frac{4}{k-1}$, \cite[Thm. 7.7, p. 222]{MR1688256} states that, 
for every $u\in H^1(\R\times\sph^k)$ we have
\[
\|u\|^{2}_{L^{2^*}(\R\times \sph^{k})} \leq \frac{4}{(k+1)(k-1)\omega_{k+1}^{\frac{2}{k+1}}} \left(\|\nabla u\|_{L^{2}(\R\times S^{3})}^{2}+\frac{(k-1)^2}{4}\|u\|_{L^{2}(\R\times \sph^{k})}^{2}\right),
\]
so that \eqref{eq:GN} holds true with
\begin{equation}\label{eq:exa_A_and_B}
A=\left(\frac{4}{(k+1)(k-1)\omega_{k+1}^{\frac{2}{k+1}}}\right)^{\frac{\theta(\alpha)}{2}}
\qquad\text{and}\qquad
B=\frac{(k-1)^2}{4}.
\end{equation}
Finally,
\[
\mu_1(\sph^k,h)=k
\qquad\text{and}\qquad
\theta(\alpha)=\frac{k+1}{2}\alpha.
\]

Taking into account Remark \ref{rem:criterionwithvolume}, we have that a sufficient condition for 
$\thresh{tr}<\thresh{ex}$ in this case is that 
\begin{equation*}
\omega_k^\frac{\alpha}{2}\rho_0^\alpha\left[ \frac{(4+\alpha )k}{4+5\alpha+2\alpha^2} \right]^{\frac{4-\alpha}{4}}
<
\frac{2+\alpha}{AB^{\theta(\alpha)/2}} \left( \frac{\theta(\alpha)-2}{\theta(\alpha)}\right)^{\frac{\theta(\alpha)}{2}}\frac{2B}{(k+1)\alpha-4}.
\end{equation*}
After some simplifications we have
\begin{equation*}
\omega_k^\frac{\alpha}{2}\rho_0^\alpha\left[ \frac{(4+\alpha )k}{4+5\alpha+2\alpha^2} \right]^{\frac{4-\alpha}{4}}
<
(2+\alpha) \left( \frac{(k+1)\alpha-4}{(k-1)\alpha}\,\omega_{k+1}^{\frac{2}{k+1}}
\right)^{\frac{\theta(\alpha)}{2}}\frac{2B}{(k+1)\alpha-4}
\end{equation*}
and finally, since
\[
\frac{\omega_{k+1}}{\omega_k}=\Beta\left(\frac12,\frac{k+1}{2}\right),
\]
\begin{equation}\label{eq:suff_cond_sph}
\underbrace{\left(\frac{2}{\alpha}\,\frac{\Beta\left(\frac12,\frac{2}{\alpha}\right)}
{\Beta\left(\frac12,\frac{k+1}{2}\right)}\right)^\frac{\alpha}{2}}_{T_1}
\underbrace{\left[ \frac{(4+\alpha )k}{4+5\alpha+2\alpha^2} \right]^{\frac{4-\alpha}{4}}}_{T_2}
<
\underbrace{\left(\frac{(k+1)\alpha-4}{(k-1)\alpha}\right)^{\frac{\theta(\alpha)}{2}-1}}_{T_3}
\underbrace{\frac{k-1}{\alpha}}_{T_4}.
\end{equation}
Then we can proceed by estimating each term, recalling that we are dealing with the range 
$\frac{4}{k+1} \leq \alpha < \frac{4}{k-1}$. As a matter of fact, with refined analytic  
estimates (all almost straightforward), it is possible to see that the above condition is 
satisfied for every $\alpha$ in the range, as long as $k\ge 6$. On the other hand, such estimates 
are very long and tedious, therefore, for the sake of brevity, we report here some more rough 
estimates which allow to check the validity of \eqref{eq:suff_cond_sph} for $k\ge 9$, 
for every $\alpha$ in the range.

First, we have 
\[
\frac{2}{\alpha}\,\frac{\Beta\left(\frac12,\frac{2}{\alpha}\right)}
{\Beta\left(\frac12,\frac{k+1}{2}\right)} = \frac{\Gamma\left(\frac{2}{\alpha}+1\right)\Gamma\left(\frac{k+1}{2}+\frac12\right)}{\Gamma\left(\frac{2}{\alpha}+\frac12\right)\Gamma\left(\frac{k+1}{2}\right)}=:H\left(\frac{2}{\alpha}\right)\frac{\Gamma\left(\frac{k+1}{2}+\frac12\right)}{\Gamma\left(\frac{k+1}{2}\right)}
\]
and 
\[
 \frac{H'(x)}{H(x)}
 = \psi(x+1)-\psi\left(x+\frac12\right),
\]
where the digamma function $\psi(x)=\frac{\Gamma'(x)}{\Gamma(x)}$ is increasing for $x>0$. 
We infer that $H\left(\frac{2}{\alpha}\right)\le H\left(\frac{k+1}{2}\right)$, whence
\[
\frac{2}{\alpha}\,\frac{\Beta\left(\frac12,\frac{2}{\alpha}\right)}
{\Beta\left(\frac12,\frac{k+1}{2}\right)} \le \frac{\Gamma\left(\frac{k+1}{2}+1\right)\Gamma\left(\frac{k+1}{2}+\frac12\right)}{\Gamma\left(\frac{k+1}{2}+\frac12\right)\Gamma\left(\frac{k+1}{2}\right)}=\frac{k+1}{2},
\]
and finally
\[
T_1\le \left(\frac{k+1}{2}\right)^{\frac{2}{k-1}},
\]
and one can see that the right hand side is decreasing in $k\ge2$.

Next, we provide very rough estimates of the other terms, and the reader will easily guess how they can 
be improved. As long as $T_2$ is concerned, since $\frac{4+\alpha }{4+5\alpha+2\alpha^2}<1$,
\[
T_2\le k^{\frac{4-\alpha}{4}}\le k^{1-\frac{1}{k+1}} < k.
\]
Coming to $T_3$, since $(k-1)\alpha<4$ for every relevant $\alpha$, 
we have
\[
T_3 \ge \left(\frac{(k+1)\alpha-4}{4}\right)^{\frac{\theta(\alpha)}{2}-1} =
\left(\frac{\theta(\alpha)}{2}-1\right)^{\frac{\theta(\alpha)}{2}-1}\ge \min_{x>0} x^x = e^{-1/e}.
\]
Finally, 
\[
T_4 = \frac{k-1}{\alpha} \ge \frac{(k-1)^2}{4}.
\]

Resuming, using $T_2<k$, a first sufficient condition for the validity of \eqref{eq:suff_cond_sph}, 
for every $\alpha$ in the range, is that $k$ satisfies
\[
\left(\frac{k+1}{2}\right)^{\frac{2}{k-1}}   < \frac{ e^{-1/e}}{4}\cdot \frac{(k-1)^2}{k},
\]
and one can easily check that this holds as long as $k\ge 11$. An easy improvement, for $k\le 10$, 
is to use the estimate $T_2<k^{1-1/11}=k^{10/11}$, getting 
\[
\left(\frac{k+1}{2}\right)^{\frac{2}{k-1}}   < \frac{ e^{-1/e}}{4}\cdot \frac{(k-1)^2}{k^{10/11}},
\]
which extends the validity to the values $k=9,10$, again for every $\alpha$. As we mentioned, 
with more precise estimates it is possible also to include the range $6 \le k \le 8$: one 
just needs to improve the estimate of both $T_3$ and mainly of $T_2$, as we just did and also 
taking into account the factor $\frac{4+\alpha }{4+5\alpha+2\alpha^2}$ which is much smaller 
than $1$ when $k$ is small.

Moreover, by continuity, in case $\alpha=\frac{4}{k+1}$ the condition \eqref{eq:suff_cond_sph} reduces to
\begin{equation}\label{eq:suff_cond_sph_mass_crit}
\left(\frac{k+1}{2}\right)^\frac{2}{k+1} \left[\frac{k(k+1)(k+2)}{14+7k+k^2} \right]^{\frac{k}{k+1}}
<
\frac{(k-1)(k+1)}{4},
\end{equation}
which holds true whenever $k\ge3$, as one can check directly (by the previous discussion, it is enough
to check it for $3\le k\le 5$, or $3\le k \le 8$). In particular, we obtain that 
$\thresh{tr}<\thresh{ex}$ whenever $k\ge3$ and $\alpha$ is slightly mass supercritical.

Finally, to complete the proof of Proposition \ref{prop:sphere}, we are left to check the 
cases $k=4,5$, for every $\alpha$ in the range. To this purpose, Remark 
\ref{rem:criterionwithvolume} is not enough, and one has to resort to Remark 
\ref{rem:improvedcriterionwithvolume}, checking the validity of the condition therein. Since 
we are left with two cases for $k$, the exponent $\alpha$ ranges in a bounded interval, and the functions 
involved in Remark \ref{rem:improvedcriterionwithvolume} are smooth in the considered range, 
such check can be performed with the help of a Computer Algebra System. As a matter of fact, the 
check is positive when $k=4,5$, while it fails for $k=3$ and $\alpha$ large, and for $k=2$ (any $\alpha$).

This concludes the proof of Proposition \ref{prop:sphere}.

\bigskip

\textbf{Acknowledgments.} Work partially supported by: PRIN-20227HX33Z ``Pattern formation in nonlinear 
phenomena'' - funded  by the European Union-Next Generation EU, Miss. 4-Comp. 1-CUP D53D23005690006; 
the Portuguese government through FCT/Portugal under the project PTDC/MAT-PUR/1788/2020; 
the MUR grant Dipartimento di Eccellenza 2023-2027; 
the INdAM-GNAMPA group.

\bigskip

\textbf{Data Availability.} Data sharing not applicable to this article as no datasets were generated or analyzed during the current study.

\bigskip

\textbf{Disclosure statement.} The authors report there are no competing interests to declare.

\bibliography{normalized}{}
\bibliographystyle{abbrv}
\medskip
\small

\begin{flushright}
{\tt dario.pierotti@polimi.it}\\
{\tt gianmaria.verzini@polimi.it}\\
{\tt junwei.yu@polimi.it}\\
Dipartimento di Matematica, Politecnico di Milano\\
piazza Leonardo da Vinci 32, 20133 Milano, Italy.
\end{flushright}

\end{document}